\documentclass[10pt]{amsproc}

\usepackage[T2A]{fontenc}
 \usepackage[english]{babel}

\usepackage{amsmath}
\usepackage{amsfonts,amssymb}
\usepackage{mathrsfs}
\usepackage{graphicx}

\usepackage{tabularx}

\theoremstyle{plain} 
\newtheorem{lemma}{Lemma}
\newtheorem{theorem}{Theorem}

\newtheorem{corollary}{Corollary}

\theoremstyle{definition}
\newtheorem{remark}{Remark}

\usepackage[matrix,arrow,curve]{xy}
\usepackage{enumitem}

\newcounter{NN}\numberwithin{NN}{section}
\renewcommand{\theNN}{{\rm \arabic{NN}${}^o$}}
\def\nr{\refstepcounter{NN}{\theNN}}%
\newcommand{\OOO}{{\mathscr O}}
 
\newcommand{\PP}{{\mathbb P}} 
\newcommand{\QQ}{{\mathbb Q}} 
 
\newcommand{\ZZ}{{\mathbb Z}}
\newcommand{\Pic}{\operatorname{Pic}}
\newcommand{\rk}{\operatorname{rk}}

\newcommand{\h}{\operatorname{h}}

\newcommand{\Cl}{\operatorname{Cl}}

\begin{document}


\date{\\
\\}

\author{Yu. Prokhorov}\thanks{
This work is supported by the Russian Science Foundation under grant 
14-50-00005. 
}
\address{
 Steklov Mathematical Institute of Russian Academy of Sciences (Moscow)
}
\email{prokhoro@mi.ras.ru}

\title[On the number of singular points]{
On the number of singular points of factorial terminal Fano threefolds
}

\maketitle





\section{Introduction}
Throughout this paper by $X$ we denote a Fano threefold with terminal 
$\QQ$-factorial singularities
and $\rk \Pic (X)=1$. 
Let $\iota=\iota (X)$ be its \textit {Fano index} and $g:=(- K_X)^3/2+1$ be its 
\textit {genus}. 
When $\iota (X) \ge 2$, more convenient invariant of a Fano threefold $X$ is its 
\textit{degree}
$d:=(- K_X)^3/\iota^3$. 
Recall that any $X$ as above has a one-parameter smoothing $\mathfrak X$ 
\cite{Namikawa-1997}. 
Let $\h^{1,2}(\mathfrak X_s)$ be the Hodge number of a general fiber of
$\mathfrak X$. 
Denote by $s(V)$ be the number of singular points of $V$. 

The upper bounds of $s(X)$ are interesting for 
classification problems, in particular, for the classification of finite 
subgroups
of the space Cremona group
(see., e.g., \cite{Prokhorov2009e}, \cite{Prokhorov-Shramov-J-const},
\cite{Prokhorov-Shramov-p-groups}, \cite{PrzyjalkowskiShramov2016}). 
Y. Namikawa in \cite{Namikawa-1997} proved the inequality
\[
s(X) \le 20- \rk \Pic (X)+\h^{1,2}(\mathfrak X_s),
\]
which holds for any Fano threefold $X$ with terminal
Gorenstein singularities. However, this estimate is quite rough. 
For Fano threefolds with \textit {non-degenerate} singularities it is 
known the inequality $s(X) \le \h^{1,2}(\mathfrak X_s)$
(see, e.g., \cite[\S 10]{Prokhorov-GFano-1}). 

\begin{theorem} [(cf. {\cite[Thm. 4.1]{Prokhorov-planes}})]
\label{main}
Suppose either 1) $\iota=1$ and $g \ge 7$ or 2)
$\iota=2$ and $d \ge 3$. Then $s(X) \le \h^{1,2}(\mathfrak X_s)$
and this bound is attained for some $X$ having only ordinary double 
points. 
\end{theorem}
Recall (see \cite{Iskovskikh-Prokhorov-1999}, see also the references in 
\cite{Przyjalkowski-Shramov-Hodge-2015}) that
$\h^{1,2}(\mathfrak X_s)$ takes the following values:
\medskip \noindent
\begin{center} \renewcommand {\arraystretch} {1.3}
\scalebox {0.85} {
\noindent
\begin{tabular} {lc||c|c|c|c|c|c|c|c|c|c}
$\iota=1$ & $g $ & 12 & 10 & 9 & 8 & 7 & 6 & 5 & 4 & 3 & 2
\\
\hline
& $\h^{1,2}$ & 0 & 2 & 3 & 5 & 7 & 10 & 14 & 20 & 30 & 52
\end{tabular} \quad
\begin{tabular} {lc||c|c|c|c|c}
$\iota=$ 2 & $d $ & 5 & 4 & 3 & 2 & 1
\\
\hline
& $\h^{1,2}$ & 0 & 2 & 5 & 10 & 21
\end{tabular}
}
\end{center}
\medskip

Note that in our theorem we do not assume that the singularities of $X$ are
non-degenerate. 
Moreover, our construction allows to get better bounds in
degenerate cases. 
If $\iota=1$, $g=6 $ or $\iota=d=2$, then under the additional 
assumption
that the variety $X$ has a $cA_1$-point,
our computations allow us to get the estimate $s(X) \le15$ (cf. \cite[proof of 
1.3]{Prokhorov-planes}). 
However, we do not claim that this bound is sharp. 

\textit {A Sarkisov link} is the following diagram
\begin{equation}
\label{equation-cd} 
\vcenter {
\xymatrix {&& Y \ar [dl]_f \ar [dr]_{\bar f} 
\ar @ {->} [rr]^{\chi} && Y \ar [dr]^{f'} \ar [dl]^{\bar f'} 
\\
& X && \bar Y && Z}} 
\end{equation} 
where $f$ and $f'$ are extremal Mori
contractions and $\chi$ is a flop. In our situation, the morphism $f$ is 
birational. Let $E$ be the exceptional divisor and $E'\subset Y'$ be its proper 
transform. If $f'$ is also birational, then we denote by $D'$ 
its exceptional divisor and $\Gamma:=f'(D')$ If, moreover, $\Gamma$ is 
curve, then it is irreducible, contained in the non-singular locus of $Z$, 
$f'$ 
is the blowup of $\Gamma$ and the singularities $\Gamma$ are planar (if there 
are any, see \cite{Cutkosky-1988}). 

\section {Birational transformations in singular points. }

Recall that a three-dimensional singularity $P\in X$ is 
\textit{of type $cA_1^m$}, if it is analytically equivalent to a hypersurface 
singularity 
$x_1x_2+x_3^2+x_4^m$. The blowup $f: Y \to X$ of the the maximal ideal 
$\mathfrak m_{P, X}$ of such a point is a Mori contraction (see 
\cite{Cutkosky-1988}). 
If the $m \le 2$, then $Y$ is smooth along $f^{-1}(P)$. 
If $m>2$, then $Y$ has on $f^{-1}(P)$ exactly one singular 
point and this point is of type $cA_1^{m-2}$.

Sarkisov links with centers at 
singular points $P \in X$ of type $cA_1$ appeared earlier in 
different papers (see, e.g., \cite{Cutrone-Marshburn}, 
\cite{Jahnke-Peternell-Radloff-II}, \cite[\S 
3]{Przhiyalkovskij-Cheltsov-Shramov-2015}, 
\cite{Takeuchi-2009}) but usually 
under some additional restrictions. My systematize and generalize this 
information: 

\begin{theorem} 
\label{cA1} 
Suppose that $X$ has a $cA_1$-point $P \in X$. Let $f: Y \to X$ be 
the blowup of $\mathfrak m_{P, X}$. 
If $\iota=2$, $d \ge 2$, then $f$ is included in the Sarkisov link 
\eqref{equation-cd}, where $\chi$ is an isomorphism and $f'$ is defined by the 
linear system $|\frac 12 (-K_{Y'}-E')|$. Except for the case \ref{d=2}, the 
morphism $f'$ contracts a divisor $D'$ to a curve $\Gamma$:
\par\medskip\noindent
\begin{tabularx}{\textwidth}{|c|c|X|c|c|}
\hline
\textnumero&d &\centering $Z$& $\deg \Gamma$& $p_a(\Gamma)$ 
\\ \hline
\nr \label{d=4} & $ 4$ & a smooth quadric in $\PP 
^4$ & $ 4$ & $ 1$\\
\nr \label{d=3} & $ 3$ & $\PP^3$ & $ 6 $ & $ 
4$\\\hline && \multicolumn {3} {c|} {} \\[- 9pt] 
\nr \label{d=2} & $ 
2$ & \multicolumn {3} {l|} {$\PP^2$, \qquad $f'$ is a conic bundle 
with discriminant curve $\Delta \subset \PP^2$, $\deg \Delta=6 
$} \\\hline 
\end{tabularx}

\par\medskip\noindent
If $\iota=1$, $g\ge 4$, then $f$ is included in the Sarkisov link 
\eqref{equation-cd}, 
where, except for the case \ref{g=4}, the morphism $f'$ is defined by the 
linear 
system $|-K_{Y'}-E'|$, and, except the cases \ref{g=6} and \ref{g=5}, $f'$ 
contracts 
a divisor 
$D'$ to a curve $\Gamma$: 

\par\medskip\noindent
\begin{tabularx}{\textwidth}{|c|c|X|c|c|}
\hline
\textnumero&g &\centering $Z$& $\deg \Gamma$& $p_a(\Gamma)$
\\\hline
\nr \label{g=10} & $ 10$ & a non-singular Fano threefold with $\iota=2$, $d 
=5$ & $ 6 $ & $ 1$\\
\nr \label{g=9} & $9$ & a $\QQ$-factorial Fano 
threefold with $\iota=2$, $d=4$ & $4$ & $ 0$
\\
\nr \label{g=8} & 
$ 8 $ & a smooth quadric in $\PP^4$ & $ 8 $ & $ 4$
\\
\nr \label{g=7} 
& $ 7$ & $\PP^3$ & $ 8 $ & $ 6 $
\\\hline && \multicolumn {3} {c|} {} 
\\[- 9pt] 
\nr \label{g=6} & $ 6 $ & \multicolumn {3} {l|} 
{$\PP^2$, \qquad $f'$ is a conic bundle with discriminant curve 
$\Delta \subset\PP^2$, $\deg \Delta=6 $} 
\\
\nr \label{g=5} & $ 5$ &\multicolumn {3} {l|} {$\PP^1$, \qquad $f'$ is 
a del Pezzo fibration of degree $ 4$} 
\\\hline && \multicolumn {3} {c|} {} 
\\[- 9pt] 
\nr \label{g=4} & $ 4$ & a $\QQ$-factorial Fano threefold with 
$\iota=1$, $g=4$ & $ 4$ & $ 0$ 
\\\hline 
\end{tabularx} 
\par \medskip \noindent
In the case \ref{g=4},\ $D'\in|-K_{Y'}-E'|$.
\end{theorem} 

\begin{proof} 
In all our cases, the linear system \mbox {$|-K_X|$} is very 
ample and defines an embedding $X \subset \PP^{g+1}$ as a variety of 
degree $2g- 2$. Consider the projection $\psi: X \dashrightarrow Y_0 \subset 
\PP^{g}$ from $P$. 
The blowup $f: Y \to X$ of the maximal ideal $\mathfrak m_{P, X}$ 
resolves the indeterminacy points of $\psi$ and gives a 
morphism $Y \to Y_0$. Note that the exceptional divisor $E \subset Y$ is
isomorphic to an irreducible quadric $Q \subset \PP^3$ and 
$\OOO_E (E) \simeq \OOO_Q (-1)$ (see
\cite{Cutkosky-1988}). It is easy to see that $K_Y=f^* K_X+E$. 
This shows that the morphism $Y\to Y_0$ is 
given by the anticanonical linear system and $-K_Y^3=-K_X^3-2>0$. Therefore, 
$\dim Y_0=3$ and in the Stein factorization $Y \to \bar Y \to Y_0$, the morphism 
$\bar f: Y\to\bar Y$ is birational and crepant. Since there are at most a finite number of lines
passing through $P$, the exceptional locus of the projection $\psi$ is 
at most one-dimensional. Thus, $\bar f: Y \to \bar Y$ is a small 
crepant contraction. Hence there exists a flop $\chi: Y \dashrightarrow Y'$ and 
an
extremal Mori contractions $f': Y' \to Z$. Using the restriction exact sequences to
$E \simeq Q \subset \PP^3$, it is easy to compute that 
$\dim|-K_{Y'}-E'|\ge \dim|-K_X|=g-4$ always and $|\frac 12 (-K_{Y'}-E')|\ge 
d$ for $\iota=2$. 
Using this and solving the 
corresponding Diophantine equations, as in \cite[ch. 4]{Iskovskikh-Prokhorov-1999} 
and \cite[\S 4]{Prokhorov-planes}, we obtain the 
possibilities \ref{d=4}-\ref{g=4}. The smoothness of $Z$ in the cases \ref{d=4}, 
\ref{g=10} and \ref{g=8} follows from $\QQ$-factoriality (see Corollary \ref{corollary-iota=2} below). 
If $\iota=2$, then the threefold $X$ does not contain lines (for the
anticanonical embedding). Therefore, the divisor $-K_Y$ is ample and $Y\simeq Y'$. 
This proves Theorem \ref{cA1}. 
\end{proof} 

\begin{remark} 
\label{remark-m} 
Let the pair $(Z, \Gamma)$ be such as in the cases 
\ref{d=4}, \ref{d=3}, \ref{g=10}-\ref{g=7}. Suppose that 

\begin{enumerate} 
\item \label{remark-m-1} 
the curve $\Gamma$ has planar singularities and is contained in 
the non-singular locus of $Z$, 

\item \label{remark-m-2} 
$\Gamma$ is a 
scheme-theoretic intersection of elements of the linear system $|\iota (Z) L|$, 
where $L$ is the ample generator of $\Pic(Z)\simeq \ZZ$, 

\item \label{remark-m-3} 
the number of $\iota(Z)$-secant lines of $\Gamma$ is at most finite. 
\end{enumerate} 
Then the blowup of $\Gamma$ can be
completed to a Sarkisov link with the corresponding Fano threefold $X$ (i.e. the construction
\eqref{equation-cd} can be reversed). 
Indeed, it follows from the condition \ref{remark-m-1} 
that $Y'$ has only terminal factorial singularities, \ref{remark-m-2} 
that $|-K_{Y'}|$ has no base points, and \ref{remark-m-3} guarantees that the 
corresponding morphism defines a crepant small contraction. The rest is similar to the 
proof of Theorem \ref{cA1}. 
\end{remark} 

\begin{corollary} \label{corollary-cA1} 
Suppose that in the conditions of Theorem \textup{\ref{main}}, the 
threefold $X$ has a point of type $cA_1$. Then 
$s(X)\le\h^{1,2}(\mathfrak X_s)$ and this bound is attained for some $X$ having 
only ordinary double points. 
\end{corollary} 

\begin{proof} 
Note that 
three-dimensional terminal flops preserve types of singularities. Therefore, in 
the 
notation of Theorem \ref{cA1} and \eqref{equation-cd} we have 
\begin{equation} 
\label{equation-n} 
s(X) \le s(Y)+1=s(Y')+1=s(Z)+s(\Gamma)+1 \le s(Z)+p_a (\Gamma)+1 
\end{equation} 
(see \cite[5.1(ii)]{Prokhorov-planes}). In all cases \ref{d=4}-\ref{g=7} we 
can choose $Z$ and $\Gamma$ so that equalities are attained. 
This proves Corollary \ref{corollary-cA1}. 
\end{proof} 

\section {Proof of Theorem \ref{main}: the case $g \neq 7$} 
\label{section-g-neq-7} 
For $\iota=1$, $g \ge 9$, 
Theorem \ref{main} was proved in \cite[1.3]{Prokhorov-planes}, and for
$\iota=1$, 
$g=8$ and $\iota=2$, $d=3$ the assertion follows directly from
Corollary \ref{corollary-cA1} and the following lemma.

\begin{lemma} 
\label{lemma-g=8} 
Assume either: 1) $\iota=1$ and $g=8$, or 2) 
$\iota=2$ and $d=3$. If $s(X)>2$, then $X$ has a singularity of 
type $cA_1$. 
\end{lemma} 

\begin{proof} 
In the case $g=8$, according to \cite[4.1]{Prokhorov-planes}, the
double projection from a line contained in the 
non-singular locus determines a Sarkisov link, where $f'$ is a conic bundle 
over $\PP^2$ with discriminant curve $\Delta$ of degree $5$. Similarly, for a 
cubic hypersurface with $\QQ$-factorial terminal singularities the usual 
projection from a line defines a Sarkisov link, where $f'$ is conic bundle 
over $\PP^2$ with discriminant curve $\Delta$ degree $5$ (in this case $\chi$ is isomorphism). 
Moreover, $Y'$ has the same collection of singularities as $X$. 
Assume that $s(X)>2$ and all the singularities of $X$ are worse than $cA_1$. 
Let $R\in Y'$ be a singular point and $o:=f'(R)$. 

We claim that $R$ is the only singular point of $Y'$ lying on the fiber $f:={f'}^{-1}(o)$ 
and the singularity of the curve $\Delta$ at $o$ is of multiplicity $\ge 3$. 
Let $U \subset Z=\PP^2$ be a small analytic neighborhood of $o$ and 
$Y'_U:={f'}^{-1}(U)$. Then $Y_U'$ can be embedded to $\PP^2_{x_0, 
x_1, x_2} \times U_{u_1, u_2}$ and defined there by a homogeneous equation of degree 
$2$ with respect to the variables $x_0, x_1, x_2$. We may assume that $R$ 
has coordinates $(0: 0: 1; 0,0)$. 

If $F$ is  a reducible conic, then the 
equation of $Y'_U$ can be written in the form  $x_0x_1+\alpha (u_1, 
u_2) x_2^2=0$, where $\alpha \in \mathfrak m_{o, U}^3$. In this 
case $R$ is the only singular point of $Y'_U$ and the discriminant curve 
$\Delta=\{\alpha=0\}$ has at $o$ a singularity of multiplicity $\ge 3$. 

If $F$ is a double line, then the equation of $Y'_U$ can be written in the form 
$x_0 ^2+\alpha x_1^2+\beta x_1x_2+\gamma x_2^2=0$, where 
$\alpha \in \mathfrak m_{o, U}$, $\beta, \gamma \in \mathfrak m_{o, U}^2$. 
Since singularities of $Y'_U$ are isolated, we have $\alpha \notin \mathfrak m_{o, U}^2$. 
This means that $R$ is the only singular point of $Y'_U$. The discriminant 
curve $\Delta$ is given by $\beta^2=4 \alpha \gamma $ and again has at $o$ 
a singularity of multiplicity $\ge 3$. This proves our assertion. 

Thus, $s(Y')$ is less or equal to the number of points of multiplicity $\ge 3$ 
of the curve $\Delta$. On the other hand, a plane (possibly reducible) curve of 
degree 5 has at most two triple points. Lemma \ref{lemma-g=8} is proved. 
\end{proof} 

It remains to consider the case $\iota=2$, $d \ge 4$. The 
following lemma is well known in the non-singular case (see, e.g., 
\cite{Iskovskikh-Prokhorov-1999}). A singular case is similar. 

\begin{lemma} 
Let $X=X_d \subset \PP^{d+1}$ be a Fano threefold with $\iota=2$, $d \ge 4$ and 
let $l$ be a line contained in the non-singular locus of $X$. Then the
projection from $l$ determines a Sarkisov link, where $\chi$
is an isomorphism. Moreover, 

\begin{enumerate} 
\item 
if $d=4$, then $Z \simeq \PP^3$, and $f'$ is the blowup of an (irreducible) curve 
of degree $5$ and arithmetic genus $2$ lying on a quadric; 
\item 
if $d=5$, then $Z$ is smooth quadric in $\PP^4$ and $f'$ is the blowup of a smooth 
rational 
cubic curve lying on a hyperplane section.
\end{enumerate} 
\end{lemma} 

\begin{corollary} 
\label{corollary-iota=2} 
If $\iota=2$, $d=5$, 
then $X$ is smooth. Let $\iota=2$, $d=4$. Then $s(X) \le 2$ and if 
$s(X)=2$, then both singularities of $X$  are of type $cA_1$. 
\end{corollary} 

\section {The case $g=7$} 

\begin{theorem} 
\label{theorem-conic} 
Let $\iota=1$, $g=7$ and let $C \subset X$ be a sufficiently general conic. 
Then there exists a Sarkisov link \eqref{equation-cd} with center $C$ and
there are two possibilities: 
\begin{enumerate} 
\item \label{theorem-conic-1} 
$Z \subset \PP^4$ is a smooth quadric, $f'$ is the blowup of a curve 
$\Gamma\subset Z$ with
$\deg\Gamma=10$, $p_a(\Gamma)=7$.
\item \label{theorem-conic-2} 
$Z\subset\PP^4$ is a cubic and $f'$ is the blowup of 
a smooth rational curve $\Gamma \subset Z$ of degree $4$.
\end{enumerate} 
\end{theorem} 

The proof  substantially is similar to \cite[\S 4.4]{Iskovskikh-Prokhorov-1999}
but in the singular case some of the arguments must 
be modified. According to \cite{Namikawa-1997} and 
\cite[4.2.5]{Iskovskikh-Prokhorov-1999} the family $\mathscr C$ parametrizing non-degenerate 
conics on $X$ is not empty and  two-dimensional, and according to 
\cite[A.1.2]{Kuznetsov-Prokhorov-Shramov} conics from the family $\mathscr C$ 
cover an open subset in $X$. Further, you need the following. 

\begin{lemma} 
\label{lemma-conics} 
\begin{enumerate} 
\item \label{lemma-conics-line} 
For 
each line $l\subset X$ there are only a finite number of lines on $X$ 
meeting $l$.
\item \label{lemma-conics-1} 
Let $l$ be a line 
contained in the non-singular locus of $X$. Then it meets a 
one-dimen\-sional family of conics. 
\item \label{lemma-conics-2} 
There is a 
non-degenerate conic contained in the non-singular locus of $X$.
\end{enumerate} 
\end{lemma} 

\begin{proof} 
\ref{lemma-conics-line} follows from 
the fact that $\dim|H-2l|\ge 1$ and $\Pic (X)=\ZZ \cdot [H]$. 

\ref{lemma-conics-1} 
Consider the double projection from $l$ and the 
corresponding Sarkisov link \cite{Prokhorov-planes}. 
In the notation of \eqref{equation-cd}, the morphism $f$ is the blowup of $l$, 
$Z \simeq\PP^1$ and 
$f'$ is a del Pezzo fibration  of degree 5. The proper transform 
$C' \subset Y'$ of a conic $C \subset X$ meeting  $l$  is a line in a fiber of 
$f'$ (because $-K_{Y'} \cdot C'=1$ and $E' \cdot C'\ge 1$). There is 
at most a one-dimensional family of such lines. This proves 
\ref{lemma-conics-1}. 

\ref{lemma-conics-2} Assume that all of the conics from the 
(two-dimensional) family $\mathscr C$ pass through a point $P \in X$. 
Take line $l \subset X$ that  does not pass through $P$ and 
all the lines $l_i$ meeting $l$ also do not pass through $P$. 
Then the map $\chi \circ f^{-1}: X \dashrightarrow Y'$ is an isomorphism near $P$. 
The line $l$ intersects a one-dimensional family of conics from the family 
$\mathscr C$ and proper transforms these conics on $Y'$ are lines in the 
fiber of $f'$ passing through the point $P':=\chi (f^{-1}(P))$. 
Hence this fiber is a cone in $\PP^5$ with vertex $P'$. But since the
singularities $Y'$ are hypersurface, the dimension of the tangent space to a fiber of 
$f'$ is at most 4. The contradiction proves our lemma. 
\end{proof} 
Thus there exists a conic $C$ contained in the non-singular locus of $X$. 
Take such a conic so that it is not contained in the surface swept out 
by lines. Consider the projection of $\psi: X \dashrightarrow Y_0 \subset 
\PP^{5}$ 
from the linear span of $C$. Since  $X \cap \langle C \rangle=C$, the 
blowup $f: Y \to X$ of the conic $C$ resolves the indeterminacy points of
$\psi$ and 
defines a morphism $f_0: Y \to Y_0$. Here in the Stein factorization $Y \to \bar Y 
\to Y_0$, the morphism $\bar f: Y \to \bar Y$ is birational and is given by a  multiple
of the linear system $|-K_Y|$, i.e., $\bar f$ is crepant. 

\begin{lemma} \label{lemma-conic-projection}
The exceptional locus of $\bar f$ is one-dimensional. 
\end{lemma}

\begin{proof} 
Suppose that $\bar f$ contracts a divisor $D$. It 
is easy to compute (see \cite[4.1.2]{Iskovskikh-Prokhorov-1999})
that $D \sim \alpha (2 (-K_Y) -3E)$ for some $\alpha \in \ZZ_{>0}$ 
and $D \cdot E \cdot (-K_Y)=14 \alpha>0$. In particular, 
$\Lambda:=f_0 (D \cap E)$ is a curve. Since the conic $C$ meets only a finite 
number 
of lines, a general fiber of $D$ must be a conic meeting $C$ at two 
points \cite[4.4.1]{Iskovskikh-Prokhorov-1999}. If the restriction 
$f_0|_{E}: E \to f_0 (E)$ is a birational morphism, $f_0 (E)$ is a ruled surface 
of degree 4 singular along $\Lambda$, where 
\[
\deg \Lambda=\frac 12D \cdot E \cdot (-K_Y)= 7 \alpha \ge 7 
\] 
This is impossible, see e.g.
\cite[4.4.8]{Iskovskikh-Prokhorov-1999}. Therefore $f_0$ is a morphism of 
degree 2, $f_0 (E)$ is a quadric, and $X_0 \subset \PP^5$ is a subvariety of 
degree 3. Since $\rk \Cl (\bar Y)=\rk \Cl (Y_0)= 1$, the variety $Y_0$ is a cone over 
a rational twisted cubic curve with a vertex in a line. However, this cone 
contains no quadrics. The contradiction proves Lemma \ref{lemma-conic-projection}.
\end{proof} 

Further, as in \cite[\S 4.1]{Iskovskikh-Prokhorov-1999}, there exists a flop 
$\chi: Y \dashrightarrow Y'$ 
and an extremal Mori contraction $f': Y' \to Z$. By solving the corresponding 
Diophantine equations, as in \cite[ch. 4]{Iskovskikh-Prokhorov-1999} or 
\cite [\S 4]{Prokhorov-planes}, we obtain the possibilities 
\ref{theorem-conic-1}-\ref{theorem-conic-2}. Theorem \ref{theorem-conic} 
is proved. For the proof of Theorem \ref{main} in the case $g=7$, similar to 
\eqref{equation-n}, we write 
\begin{equation*} 
\label{equation-s}
s(X)=s(Y) =s(Y')=s(Z)+s(\Gamma) \le s(Z)+p_a (\Gamma). 
\end{equation*} 
If the possibility \ref{theorem-conic-1} of Theorem \ref{theorem-conic} 
occurs, 
then the variety $Z$ is smooth and therefore $s(X)\le p_a(\Gamma)=7$. 
Otherwise the curve $\Gamma$ is smooth and therefore $s(X)=s(Z)$. 
Then $s(Z)\le5$ according to the already proved case of 
Theorem \ref{main}. 

 
 \def\cprime{$'$}

\end{document}